\documentclass[12pt,a4paper,requno]{amsart}
\usepackage{layout}
\usepackage{anysize}

\marginsize{22mm}{18mm}{17mm}{18mm}
\usepackage{todonotes}
\usepackage{amsmath,amssymb,amsthm,amscd,amsfonts}
\usepackage{bbm}
\usepackage{mathtools}
\usepackage{mathrsfs}
\usepackage{hyperref}
\usepackage{graphics}
\usepackage{enumerate}
\usepackage{comment}
\theoremstyle{definition}
\newtheorem{rmk}{Remark}

\theoremstyle{plain}
\newtheorem{thm}[rmk]{Theorem}
\newtheorem{mthm}[rmk]{Main Theorem}
\newtheorem{lem}[rmk]{Lemma}

\newtheorem{prop}[rmk]{Proposition}
\newtheorem{rem}[rmk]{Remark}

\DeclareMathOperator{\divi}{div}
\DeclareMathOperator{\vol}{vol}

\DeclareMathOperator{\Sp}{Sp}
\DeclareMathOperator{\Sym}{Sym}
\DeclareMathOperator{\ima}{im}
\DeclareMathOperator{\Ima}{Im}
\DeclareMathOperator{\Mat}{Mat}
\DeclareMathOperator{\GL}{GL}
\newcommand{\Ab}{\mathcal A}
\newcommand{\dd}{\mathrm d}
\newcommand{\ddc}{\mathrm{d}^c}
\newcommand{\eps}{\varepsilon}
\newcommand{\MH}{\mathbbm{H}}
\newcommand{\MZ}{\mathbbm{Z}}
\newcommand{\MN}{\mathbbm{N}}
\newcommand{\MC}{\mathbbm{C}}
\newcommand{\MRE}{\mathbbm{R}}

\newcommand{\dell}{\partial}

\DeclareMathOperator{\Pet}{Pet}
\DeclareMathOperator{\Spec}{Spec}

\newtheorem*{set*}{\TagSymbol}
\providecommand{\TagSymbol}{}


\title{The arithmetic volume of the moduli space of abelian surfaces}
\author{Barbara Jung} 
\address{\rm Institut f\"ur Mathematik, Humboldt-Universit\"at zu Berlin, Unter den Linden 6, 10099 Berlin, Germany}
\email{beijung11@gmail.com}
\author{Anna-Maria von Pippich}
\address{\rm Fachbereich Mathematik und Statistik, Universit\"at Konstanz, Universit\"atsstra{\upshape{\ss}}e 10, 78464 Konstanz, Germany}
\email{anna.pippich@uni-konstanz.de}

\begin{document}
\setcounter{tocdepth}{1}
\setcounter{section}{0}
\maketitle
\begin{abstract}
Let $\Ab_g$ denote the moduli stack of principally polarized abelian varieties of dimension $g$.
The arithmetic height, or arithmetic volume, of $\overline{\Ab}_g$, is defined to be the arithmetic 
degree of the metrized Hodge bundle $\overline{\omega}_g$ on $\overline{\Ab}_g$.
In 1999, K\"uhn proved a formula for the arithmetic volume of $\overline{\Ab}_1$
in terms of special values of the Riemann zeta function. In this article, we generalize his result to the case $g=2$.
 \end{abstract}

\section{Introduction}

\subsection{Arithmetic volumes}
Let $\Ab_g$ denote the moduli stack of principally polarized abelian varieties of dimension $g$.
In 1943, Siegel \cite{siegel} inductively computed the geometric volume, i.e., the degree of the associated Hodge bundle $\omega_g$, as
$$
\vol(\Ab_g):=\deg(\omega_g)=(g-1)!\pi^{-g}\zeta(2g)\vol(\Ab_{g-1}),
$$
where $\zeta(\cdot)$ denotes the Riemann $\zeta$-function. \\
 A model for the moduli stack $\Ab_g$ over $\Spec(\MZ)$ was constructed by Mumford et al.~in \cite{mumford3}. Later, Faltings and Chai \cite{faltings} were able to construct a model for toroidal compactifications $\overline{\Ab}_g$ over $\Spec(\MZ)$.
An arithmetic intersection theory, i.e., an intersection theory for varieties over $\Spec(\MZ)$, or, more generally, over the spectra of arithmetic rings $A$, was first approached by Arakelov \cite{arakelov} (and later extended by Deligne \cite{deligne}) in the case of relative dimension 1. By ``compactifying'' the variety over $A$ by a complex fibre induced by the complex embeddings of $A$, the authors were able to define a good notion of intersection, considering the additional datum of Green functions induced by holomorphic vector bundles with smooth Hermitian metrics  on the complex fibre. Soon after, this approach was generalized by Gillet and Soul\'e for arbitrary relative dimension.\\
However, this theory is limited to vector bundles with smooth metrics and it can thus not be applied to the case of the metrized Hodge bundle $\overline{\omega}_g$ on $\overline{\Ab}_g$, as the $L^2$-metric acquires logarithmic singularities on the toroidal boundary. This fact was one of the motivations for Burgos, Kramer, and K\"uhn \cite{bkk} to come up with a more flexible concept of arithmetic Chow groups that can be applied to the situation of metrics having mild singularities. 
 Laying the foundation for this generalized intersection theory,
in 1999, K\"uhn \cite{kuehn}   computed the arithmetic height, or arithmetic volume, of $\overline{\Ab}_1$, which is by definition the arithmetic degree of the metrized Hodge bundle $\overline{\omega}_1$ on $\overline{\Ab}_1$, to equal
$$
\widehat{\vol}(\overline{\Ab}_1)=\widehat{\deg}(\overline{\omega}_1)=\zeta(-1)\left(\frac{\zeta'(-1)}{\zeta(-1)}+\frac{1}{2}\right).
$$
Further results by Bruinier, Burgos and K\"uhn \cite{borcherds} include the case of Hilbert modular surfaces.
A result connecting integrals over Borcherds forms with special values of Eisenstein series by Kudla \cite{kudla} is a main ingredient for the result of this article. Moreover, an inductive approach for Shimura varieties of orthogonal type was given by H\"ormann in \cite{hoermann}, which is also applicable to the problem tackled here. More general conjectures about arithmetic intersection numbers by Kramer, Maillot--Roessler \cite{maillotroessler}, and by the so-called Kudla-Program are stating connections between Fourier coefficients of certain modular forms and classes of algebraic cycles in the arithmetic Chow groups. In particular, the arithmetic degree of the bundle of modular forms is conjectured to be a rational linear combination of logarithmic derivatives of the $\zeta$-function evaluated at negative odd integers.

\subsection{Main result}
In this article, we consider the case $g=2$ and we obtain a formula for the arithmetic height, or arithmetic volume, of $\overline{\Ab}_2$, by computing the degree of the metrized Hodge bundle $\overline{\omega}_2$ on $\overline{\Ab}_2$ as  (see Theorem \ref{thm_main})
$$
\widehat{\vol}(\overline{\Ab}_2)=\widehat{\deg}(\overline{\omega}_2)=\zeta(-3)\zeta(-1)\left(2\frac{\zeta'(-3)}{\zeta(-3)}+2\frac{\zeta'(-1)}{\zeta(-1)}+\frac{17}{6}\right)+ c_2 \log 2+c_3 \log 3, 
$$
with constants $c_2, c_3\in\mathbb{Q}$ arising from the intersection number at the finite places $2$ and $3$.
We are taking an explicit approach, identifying the Hodge bundle with the bundle of modular forms, choosing certain well-investigated sections of this bundle and tracing the value of the arithmetic volume back to the results mentioned above by giving a recursive formula of integrals over cycles.

\subsection{Outline of the article}
The paper is organized as follows. In Section 2, we begin by collecting background information. In particular, 
we state an identity that allows us to decompose integrals over $*$-products of Green currents into computable parts (see Proposition \ref{prop_integrals}).
In Section 3, we apply Proposition \ref{prop_integrals} inductively to obtain an explicit expression
for the contribution of the arithmetic volume in question arising from the complex fibre (see Proposition \ref{prop_5integrals}). 
In Section 4, we treat the term $(B)$ arising in this expression, in particular the integrals over the positive codimensional cycles of $\overline{\mathcal{A}}_2$.
In Section 5, we treat the term $(A)$ arising in the above mentioned expression. 
In Section 6, we prove the main result of the paper (see Theorem \ref{thm_main}).

\subsection{Acknowledgements} 
We would like to express our thanks to J. Kramer for proposing the topic  and for his valuable advice in the course of the write-up of this article. Furthermore, we would like to thank U. K\"uhn for helpful discussions.
Both authors acknowledge support from the DFG Graduate School Berlin Mathematical School and the DFG Research Training Group \emph{Arithmetic and Geometry}, von Pippich acknowledges support from the LOEWE research unit \emph{Uniformized structures in Arithmetic and Geometry} of Technical University Darmstadt and Goethe-University Frankfurt.

\section{Background material}\label{sect:2}
In this section, we will give some background theory. First, we will introduce the moduli space $\mathcal{A}_2$ and its toroidal compactification. Then, we will give formulas for some particular Siegel modular forms living on it. Finally, we will define the notion of its arithmetic volume and state a result that will be used for its explicit computation. 
\subsection{The moduli space $\mathcal{A}_2$}\label{sect: 2.1}

The \textit{Siegel upper half-space} $\mathbbm{H}_2$ is the set
$$
\mathbbm{H}_2:=\left\{ \tau=x+iy\in \Sym_2(\mathbbm{C})\,\left|\right.\, x,y\in \Sym_2(\MRE),\, y>0\right\},
$$
where the notation $y>0$ denotes that the matrix $y=\ima(\tau)$ is positive definite. As an open submanifold of $\Sym_2(\mathbbm{C})$, it has dimension $3$. We denote the coordinates on $\MH_2$ by
\[ \tau: = \begin{pmatrix}\tau_1 & \tau_{12}\\ \tau_{12} &\tau_2\end{pmatrix}=x+iy=\begin{pmatrix}x_1 & x_{12}\\ x_{12} &x_2\end{pmatrix}+i\begin{pmatrix}y_1 & y_{12}\\ y_{12} &y_2\end{pmatrix}.
\]
Setting
$$J=\begin{pmatrix}0_2&\mathbbm{1}_2\\-\mathbbm{1}_2&0_2\end{pmatrix},
$$
the \textit{symplectic group} $\Sp_{4}(\MZ)$ is the subgroup of $\GL_{4}(\MZ)$ defined via
$$
\Sp_{4}(\MZ):=\left\{ M\in \Mat_{4}(\MZ)\,\left|\right.\, M^t J M=J \right\}.
$$
From now on, we will shortly write $\Gamma_2:=\Sp_{4}(\MZ)$.
The action of the symplectic group on the Siegel upper half-space $\MH_2$ is given by the prescription
$$
M\tau:=(A\tau+B)(C\tau+D)^{-1} \quad \left(M=\begin{pmatrix}A & B \\ C & D\end{pmatrix}\in \Gamma_2,\, \tau\in\MH_2\right).
$$
The quotient $\mathcal{A}_2:=\Gamma_2\backslash \mathbbm{H}_2$ is the moduli space of principally polarized abelian surfaces. We denote the quotient morphism by $\pi_2\colon \MH_2\longrightarrow \Ab_2$.

A particular divisor on $\mathcal{A}_2$ appearing later is the Humbert surface given as
\begin{equation}\label{eq_Humbert}
\mathcal H:=\pi_2\Biggl( \biggl\{\begin{pmatrix}\tau_1 & 0 \\ 0 & \tau_2 \end{pmatrix}\,\biggl\vert\biggr.\, \tau_1, \tau_2 \in \MH_1\biggr\}\Biggr)\cong\mathrm{Sym}_2(\mathcal{A}_1) =\mathcal{A}_1\times \mathcal{A}_1\bigl/\bigr.\bigl((\tau_1, \tau_2)\sim (\tau_2, \tau_1)\bigr),
\end{equation}
 Due to the isomorphism $\mathcal H\cong\mathrm{Sym}_2(\mathcal{A}_1)$, we will denote a subset of $\mathcal{H}$ that is a product of two subsets $S_1$, $S_2$ of $\Ab_1$ by
 \begin{equation}\label{eq_Hsubset}
 S_1\times_\mathcal{H} S_2:=\pi_2\Biggl( \biggl\{\begin{pmatrix}\tau_1 & 0 \\ 0 & \tau_2 \end{pmatrix}\,\biggl\vert\biggr.\, \tau_1\in S_1, \tau_2 \in S_2\biggr\}\Biggr).
\end{equation}
A smooth compactification $\overline{\mathcal{A}}_2$ of (a cover of) $\mathcal{A}_2$ such that the boundary is a normal crossing divisor can be achieved by a suitable toroidal compactification.
We recall that the boundary divisor of a toroidal compactification of $\mathcal{A}_2$ is isomorphic to a compactification of the universal family over $\Ab_{1}$, and, therefore, the boundary is of codimension 1, i.e., we have
\begin{align}
\label{formula_boundary_A2}\partial \mathcal{A}_2=\overline{\mathcal{A}}_2\setminus \mathcal{A}_2\cong (\MZ^{2}\rtimes \Gamma_{1})\backslash (\mathbbm{C}^2\times  \mathbbm{H}_{1})^{\overline{\phantom{a}}}. 
\end{align}
Local coordinates in a neighbourhood of $\partial \mathcal{A}_2$ are given by
$$
t:=\exp(2\pi i \tau_1),\, \tau_2=x_2+iy_2,\, \tau_{12}=x_{12}+iy_{12}.
$$

\subsection{Siegel modular forms}\label{sect:2.2}
A holomorphic function $f\colon \MH_2\longrightarrow \MC$ is called \textit{Siegel modular form of degree 2 and weight $k\in\MN$ for $\Gamma_2$} if the following conditions are satisfied:
\begin{itemize}
\item[(i)] $f(M\tau)=\det(C\tau+D)^k f(\tau)$\quad $\left(\tau\in\MH_2, \,M=\left(\begin{smallmatrix}A&B\\C&D\end{smallmatrix}\right)\in\Gamma_2\right)$,
\item[(ii)] $f$ is bounded on all subsets of the form $\left\{\tau=x+iy\in\MH_2\,\left|\right.\, y-y_0>0\right\}\subseteq \MH_2$ with $y_0\in\Sym_2(\MRE)$ positive definite.
\end{itemize}

Condition (ii) implies that $f$ is holomorphic at the boundary of $\MH_2$. Here, it actually follows automatically from the holomorphicity of $f$ and condition (i) by the Koecher principle. Siegel modular forms of degree 2 and weight $k$ obviously form a vector space, which we will denote by $M_k(\Gamma_2)$. Due to their transformation properties with respect to $\Gamma_2$, one can see that they are global sections of  a line bundle $\mathcal{M}_k(\Gamma_2)$ on $\Ab_2$. With the above observation, one obtains the identification
$$
H^0\left(\Ab_2, \omega_2^{\otimes k}\right)\cong M_k(\Gamma_2),
$$ 
where $\omega_2$ denotes the Hodge bundle on $\Ab_2$. The $L^2$-metric on the Hodge bundle induces the so-called Petersson metric on the line bundle $\mathcal{M}_k(\Gamma_2)$, defined by
\begin{align*}
\Vert f(\tau)\Vert^2_{\Pet}:=\vert f(\tau)\vert^2\bigl((4\pi)^2\det(\Ima (\tau))\bigr)^k,
\end{align*}
for $ f$ a global section of $\mathcal{M}_k(\Gamma_2)$.\\
The Siegel modular forms in question will now be constructed by means of $\vartheta$-series. The $\vartheta$-series of a vector $(a,b)\in (\MZ/2\MZ)^2\times (\MZ/2\MZ)^2$ is given by the equality
\begin{equation}\label{eq_thetaseries}\vartheta_{a,b}(\tau):=\sum_{n\in \MZ^2} e^{2\pi i\Bigl(\frac{1}{2}\bigl(n+\frac{a}{2}\bigr)^t\tau\bigl(n+\frac{a}{2}\bigr)+\bigl(n+\frac{a}{2}\bigr)^t\frac{b}{2}\Bigr)}.
\end{equation}
In the coordinates of $\MH_2$, and with $n=(n_1, n_2)^t$, $a=(a_1, a_2)^t$, $b=(b_1, b_2)^t$, this can be rewritten as
\begin{align}
\vartheta_{a,b}(\tau)=
\sum_{n_1,n_2\in \MZ} e^{\pi i\Bigl(\bigl(n_1+\frac{a_1}{2}\bigr)^2\tau_1+2\bigl(n_1+\frac{a_1}{2}\bigr)\bigl(n_2+\frac{a_2}{2}\bigr)\tau_{12}+\bigl(n_2+\frac{a_2}{2}\bigr)^2\tau_2+\bigl(n_1+\frac{a_1}{2}\bigr)b_1+\bigl(n_2+\frac{a_2}{2}\bigr)b_2\Bigr)}.
\end{align}
We note that $\vartheta_{a,b}$ is non-trivial if and only if $a^t b$ is even; we then call $\vartheta_{a,b}$ \textit{even}. There are 10 even $\vartheta$-series of degree 2.

The first Siegel modular form we will use is the cusp form $\chi_{10}$ of weight $10$. It is given (up to normalization) by the product of the squares of the 10 even $\vartheta$-series
\begin{align*}
\chi_{10}(\tau)&:=\frac{1}{2^{12}}\prod_{(a,b)\; \rm{even}}\vartheta^2_{a,b}(\tau)\\
&=e^{2\pi i (\tau_1+\tau_{12}+\tau_2)}\prod_{n,l,m\in \MZ\atop(n,l,m)>0}\Bigl(1-e^{2\pi i(n\tau_1+l\tau_{12}+m\tau_2)}\Bigr)^{2f(nm,l)}.
\end{align*}
Here, $(n,m,l)>0$ means that $m, n\geq 0$, and $l=-1$ for $m=n=0$. The exponents $f(nm,l)$ of the product expansion are given by the Fourier coefficients of a weak Jacobi form, see \cite{gritsenko}, Chapter 4. The notation $(n,m,l)>0$ means that $m, n\geq 0$, and $l=-1$ for $m=n=0$. Note that the normalization differs from the convention in the literature by the factor 4, and rather corresponds to the one of $\chi_{10}$ in \cite{igusaMFZ}. This is done in order to have integer coprime Fourier coefficients, see also \cite{ressal}. It is a well-known result, following from the product expansion, that the divisor of $\chi_{10}$ is given by
\[
\divi (\chi_{10})=\partial \mathcal{A}_2+ \overline{\mathcal{H}}.
\] 
with $\partial \mathcal{A}_2$ given by \eqref{formula_boundary_A2} and $\mathcal{H}$ as in (\ref{eq_Humbert}).

We will further consider the following three modular forms of degree 2 given by
\begin{align*}
&E_4(\tau):=\frac{1}{4}\sum_{(a,b)\; \rm{even}}\vartheta_{a,b}^8(\tau),\\
&E_6(\tau):=\frac{1}{4}\sum\pm\bigl(\vartheta_{a_1, b_1}(\tau) \vartheta_{a_2, b_2}(\tau) \vartheta_{a_3, b_3}(\tau)\bigr)^4,\\[0.2cm]
& \chi_{12}(\tau):=\frac{1}{2^{15}}\sum\bigl(\vartheta_{a_1, b_1}(\tau)\cdot \ldots\cdot\vartheta_{a_6, b_6}(\tau)\bigr)^4.
\end{align*}
 The second sum runs over all syzygous triples $(a_j, b_j)$ $(j=1,2,3)$, and the third sum runs over the complements of the syzygous quadruples as described in \cite{igusa}, Chapter 4. A triple $(m_1, m_2, m_3)$ is called \textit{syzygous} if $m_1+m_2+m_3$ is even. A quadruple $(m_1, m_2, m_3, m_4)$ is called syzygous if any triple $(m_j, m_k, m_l)$ $(1\leq j<k<l\leq 4)$ is syzygous. The signs in the second sum arise from a symmetrization process, making it a modular form for $\Gamma_2$. For details, see \cite{igusaMFZ}. Note that the normalization of $\chi_{12}$ again differs from the convention in the literature by a factor 12, and rather corresponds to the one of $X_{12}$ in \cite{igusaMFZ}, in order to have integer coprime Fourier coefficients, see again also \cite{ressal}.

If $\tau_{12}=0$, i.e., if
$$\tau=\begin{pmatrix}\tau_1&0 \\  0 & \tau_2 \end{pmatrix},$$ 
we see from (\ref{eq_thetaseries}) that the $\vartheta$-series $\vartheta_{a,b}$ with $a=(a_1, a_2)^t$ and $b=(b_1,b_2)^t$ decomposes as
\begin{align}
\vartheta_{a,b}(\tau)&= \vartheta_{a_1, b_1}(\tau_1)  \vartheta_{a_2, b_2}(\tau_2).
\end{align}
Applying this decomposition to the formulas for the Siegel modular forms $E_4$, $E_6$, $\chi_{12}$, their restrictions to $\overline{\mathcal{H}}$ decompose as 
\begin{align}\label{eq_Hreduction}
&E_4(\tau)=E_4(\tau_1)E_4(\tau_2),\quad E_6(\tau)=E_6(\tau_1)E_6(\tau_2), \quad 
\chi_{12}(\tau)=12\Delta(\tau_1)\Delta(\tau_2),
\end{align}
with the degree 1 modular forms
\begin{align*}
&E_4(\tau_1)=\frac{1}{2}\bigl(\vartheta^8_{00}(\tau_1)+\vartheta^8_{01}(\tau_1)+\vartheta^8_{10}(\tau_1)\bigr),\\[0.1cm]
&E_6(\tau_1)=\frac{1}{2}\bigl(\vartheta^4_{00}(\tau_1)+\vartheta^4_{01}(\tau_1)\bigr)\bigl(\vartheta^4_{00}(\tau_1)+\vartheta^4_{10}(\tau_1)\bigr)\bigl(\vartheta^4_{01}(\tau_1)-\vartheta^4_{10}(\tau_1)\bigr),\\[0.1cm]
&\Delta(\tau_1)=\frac{1}{2^8}\bigl(\vartheta_{00}(\tau_1)\vartheta_{01}(\tau_1)\vartheta_{10}(\tau_1)\bigr)^8=\exp(2\pi i \tau_1)\prod_{n \in \MZ}(1-e^{2\pi i n \tau_1})^{24}.
\end{align*}
Note that with this normalization we have
\begin{equation}\label{eq_E4Delta}
E^3_4(\tau_1)-E^2_6(\tau_1)=12^3\Delta(\tau_1).
\end{equation}
Similarly, one sees that for $\tau=\left(\begin{smallmatrix}i\infty & \tau_{12} \\  \tau_{12} & \tau_2 \end{smallmatrix}\right)\in \partial \mathcal{A}_2$, the $\vartheta$-series reduce to
\begin{align*}
\vartheta_{a,b}(\tau)&=  \vartheta_{a_2, b_2}(\tau_2),\;\; \text{if }(a_1, b_1)=(0,b_1), \\
\vartheta_{a,b}(\tau)&=0,\;\;\phantom{aaaaaa.}\text{if }(a_1, b_1)=(1,b_1).
\end{align*}
Hence, on $\partial \mathcal{A}_2$, the Eisenstein series restrict to
\begin{align}\label{eq_Dreduction}
& E_4(\tau)\big\vert_{\partial\Ab_2}=E_4(\tau_2), \quad E_6(\tau)\big\vert_{\partial\Ab_2}=E_6(\tau_2).
\end{align}
Note that, with the above normalizations, the Fourier coefficients around the cusp $\tau_1=i\infty$ of the modular forms $E_4$, $E_6$, $\chi_{10}$, $\chi_{12}$ are all integer and coprime. Therefore, the forms are defined over $\mathrm{Spec}(\MZ)$ by the $q$-expansion principle, see, e.g., \cite{katz}.
\vskip 2mm
\subsection{The arithmetic volume}\label{sect:2.3} 
We will define the arithmetic volume of $\overline{\mathcal{A}}_2$ as the highest self-intersection number of the Hodge bundle on $\overline{\mathcal{A}}_2$ over $\mathrm{Spec}(\MZ)$. In order to do so, we first introduce the notion of a partition of unity adapted to two cycles. \\

We consider a proper smooth real variety $X$, and $Y$, $Z$ be two cycles on $X$ of codimension $p$,  $q$ with support $\vert Y\vert$, $\vert Z\vert$, respectively. We assume that neither $\vert Y\vert$ nor $\vert Z\vert$ are contained in the intersection $\vert Y\vert\cap \vert Z\vert$. We define a partition of unity $\{\sigma_{YZ}, \sigma_{ZY}\}$ as follows: By a result of \cite{burgosgreen} there exists a resolution $\pi:\widetilde X_{\mathbbm R}\rightarrow X_{\mathbbm R}$ of singularities of $\vert Y\vert\cup \vert Z\vert$ which factors through embedded resolutions of $\vert Y\vert$, $\vert Z\vert$, $\vert Y\vert\cap\vert Z\vert$.  Denote by $\widehat Y$ the normal crossing divisor formed by the components of $\pi^{-1}(\vert Y\vert)$ that are not contained in $\pi^{-1}(\vert Y\vert \cap \vert Z\vert)$. Analogously, denote by $\widehat Z$ the normal crossing divisor formed by the components of $\pi^{-1}(\vert Z\vert)$ that are not contained in $\pi^{-1}(\vert Y\vert \cap \vert Z\vert)$. Hence, $\widehat Y$ and $\widehat Z$ are closed subsets of $\widetilde X$ that do not meet. Therefore, there exist two smooth, $F_\infty$-invariant functions $\sigma_{YZ}$ and $\sigma_{ZY}$ satisfying $0\leq\sigma_{YZ}, \sigma_{ZY}\leq 1$, $\sigma_{YZ}+\sigma_{ZY}=1$ with $\sigma_{YZ}=1$ in a neighborhood of $\widehat Y$ and $\sigma_{ZY}=1$ in a neighborhood of $\widehat Z$. \\
\begin{rem} In the following, $\{\sigma_{4,6}, \sigma_{6,4}\}$ will always denote a partition of unity adapted to the cycles $\divi (E_4)$ and $\divi(E_6)$, restricted to $X\subseteq \overline{\Ab}_2$ arising from the context.
\end{rem}

The toroidally compactified moduli stack $\overline{\Ab}_2$ is the complex fibre for an arithmetic variety over $\mathrm{Spec}(\MZ)$, constructed by Faltings and Chai \cite{faltings}. The Hodge bundle on $\overline{\Ab}_2$ is the complex bundle associated to the line bundle of invariant $2$-differentials on the universal semi-abelian scheme also constructed in \cite{faltings}. The arithmetic degree of the Hodge bundle over $\mathrm{Spec}(\MZ)$, i.e., its arithmetic self intersection number, can then be interpreted as the arithmetic volume of $\overline{\Ab}_2$ (as was done with the geometric volume in the introduction). As the Petersson metric has logarithmic singularities at the boundary $\partial \Ab_2$, one has to apply the arithmetic intersection theory established in \cite{bkk} by Burgos, Kramer and K\"uhn, generalizing the intersection theory for smooth metrics by Gillet and Soul\'{e} \cite{gilletsoule}, to compute the self intersection number of the Hodge bundle over $\mathrm{Spec}(\MZ)$. 
By this theory, we can write the arithmetic volume $\widehat{\vol}(\overline{\Ab}_2)$ as a sum 
$$\widehat{\vol}(\overline{\Ab}_2)=\widehat{\vol}(\overline{\Ab}_2)_{\mathrm{fin}}+\widehat{\vol}(\overline{\Ab}_2)_{\infty},$$
with $\widehat{\vol}(\overline{\Ab}_2)_{\mathrm{fin}}$ the self intersection number of the Hodge bundle over $\mathrm{Spec}(\MZ)$ in the sense of Serre, and the contribution of the volume coming from the complex fibre
$$
\widehat{\vol}(\overline{\Ab}_2)_{\infty}=\frac{1}{(2\pi i)}\int\limits_{\overline{\Ab}_2}g_1*g_2*g_3*g_4.
$$
Here, the $g_j=\log\Vert f_j\Vert_{\Pet}$ are Green currents corresponding to the divisors $D_1,\ldots,D_4$ four sections $f_1,\ldots,f_4$ of the Hodge bundle on $\overline{\Ab}_2$ that intersect successively properly. In the case of logarithmically singular metrics, the $*$-product of two Green currents $g_Y$, $g_Z$ corresponding to cycles $Y$, $Z$ is given by
\begin{align}\label{eq_starprod}
g_Y*g_Z
&= 4\pi i(\sigma_{ZY}g_Y)\wedge\dd\ddc g_Z+4\pi i\dd\ddc(\sigma_{YZ}g_Y)\wedge g_Z,
\end{align}
see \cite{bkk}, Section~6.1 for details.\\
Subsequently, we will denote by $g_{4}$, $g_6$, $g_{10}$, $g_{12}$ the Green currents corresponding to $E_4$, $E_6$, $\chi_{10}$, $\chi_{12}$, respectively. Furthermore, we will use the notation $\omega_{k}=4\pi i\dd\dd^c g_k$ $(k=4,6,10,12)$.
The following proposition gives an explicit way to decompose the integral over $*$-products of Green currents into computable parts.

\begin{prop}\label{prop_integrals} Let $X$ be a proper smooth real variety, and let $D$ be a normal crossing divisor on $X$. Assume that $D=D_1\cup D_2$, where $D_1$ and $D_2$ are normal crossing divisors of $X$ satisfying $D_1\cap D_2=\emptyset$.
Let $Y$ and $Z$ be cycles of codimension $p$ and $q$ with corresponding Green objects $\mathfrak{g}_Y=(\omega_Y,\widetilde{g}_Y)$ and $\mathfrak{g}_Z=(\omega_Z,\widetilde{g}_Z)$, respectively.
Assume that $p+q=d+1$, and that $Y$ and $Z$ intersect properly, i.e., $\vert Y\vert\cap \vert Z\vert=\emptyset$. Furthermore, assume $\vert Y\vert\cap D_2=\emptyset$ and $\vert Z\vert\cap D_1=\emptyset$. Then, we have the decomposition
\begin{align*}
&\frac{1}{(2\pi i)^d}\int\limits_{X} g_Y*g_Z\\
&=\lim_{\varepsilon \rightarrow 0}\left( \frac{1}{(2\pi i)^d}\int\limits_{X\backslash B_{\varepsilon}(D)}g_Y\wedge \omega_Z-\frac{2}{(2\pi i)^{d-1}}\int\limits_{\partial B_{\varepsilon}(D_1)}\left[g_Z\wedge \dd^c g_Y-g_Y\wedge \dd^c g_Z\right]\right)\\
&\phantom{aa}+\frac{1}{(2\pi i)^{q-1}}\int\limits_{\overline{Y\backslash Y\cap D_1}}g_Z.
\end{align*}
\end{prop}
\begin{proof} See \cite{borcherds}, Theorem~1.14.
\end{proof}
\vskip 2mm

\section{A formula for the complex contribution}\label{sect:3}
In this section, we will apply Proposition \ref{prop_integrals} inductively to obtain an explicit formula for the
\emph{complex contribution} of the arithmetic volume. First, in Lemma \ref{lem_properintersection}, we will check the necessary prerequisites
of  Proposition \ref{prop_integrals} are satisfied by showing that the divisors of the modular forms in question 
intersect successively properly. 
Then, in Proposition \ref{prop_5integrals} we show that the complex contribution decomposes into two terms $(A)$ and $(B)$,
where $(A)$ is the limit term arising from the first induction step, and $(B)$ is essentially the Faltings height of the Humbert surface.

\begin{lem}\label{lem_properintersection}
In the notation of Section \ref{sect:2}, we have the following intersections
\begin{align*}
& \divi (\chi_{10})\cdot \divi (E_6)
=\partial \mathcal{A}_2\cdot \divi (E_6)+\frac{1}{2}\bigl(\{i\}\times_\mathcal{H} \overline{\mathcal{A}}_1\bigr),
\\
& \divi (\chi_{10})\cdot \divi (E_6)\cdot \divi (E_4)=\frac{1}{6}(\{i\}\times_\mathcal{H}\{\omega\}),\\[0.1cm]
& \divi (\chi_{10})\cdot \divi (E_6)\cdot\divi (E_4)\cdot \divi (\chi_{12})=\emptyset.
\end{align*}
Hence, in particular, the divisors corresponding to the modular forms $\chi_{10}$, $E_6$, $E_4$, $\chi_{12}$ 
intersect successively properly. 
\end{lem}

\begin{proof}
Restricting $E_6$ to $\overline{\mathcal{H}}$ and $\partial \mathcal{A}_2$, one obtains by means of \eqref{eq_Hreduction} and (\ref{eq_Dreduction}) that
\[
\divi (E_6)\cdot \divi (\chi_{10})= \pi_2\Biggl(\biggl\{ \begin{pmatrix}i\infty&\tau_{12}\\ \tau_{12}&i\end{pmatrix}\,\biggl\vert\biggr.\, \tau_{12}\in \MC\biggr\}\Biggr)+\frac{1}{2}\bigl(\{i\}\times_\mathcal{H} \overline{\mathcal{A}_1}\bigr),
\]
as $E_6(\tau_2)$ has a zero of order $1/2$ for $\tau_2=i$. Restricting $E_4$ to $\divi (E_6)\cdot \divi (\chi_{10})$ and noting that $E_4(\tau_2)$ has a zero of order $1/3$ at $\tau_2=\omega$, one obtains by means of (\ref{eq_Hreduction}) and (\ref{eq_Dreduction}) that
\[\divi (E_4)\cdot \divi (E_6)\cdot \divi (\chi_{10})=\frac{1}{6}(\{i\}\times_\mathcal{H}\{\omega\}).
\]
Finally, as $\chi_{12}$ vanishes only at the boundary $\partial \mathcal{A}_2$, one sees that
 \[\divi (\chi_{12})\cdot \divi (E_4)\cdot \divi (E_6)\cdot \divi (\chi_{10})=\emptyset, 
\]
so the divisors intersect properly.
\end{proof}
\begin{prop}\label{prop_5integrals} 
The contribution of the arithmetic volume arising from the complex fibre can be computed as
\begin{align*}
&\widehat\vol(\overline{\mathcal{A}}_2)_{\infty}=\frac{1}{(2\pi i)^3}\int\limits_{\overline{\mathcal{A}}_2} g_{10}*g_{6}*g_{4}*g_{12}=(A)+(B),
\end{align*}
where we have set
\begin{align*}
(A)&:=\lim_{\varepsilon \rightarrow 0}\left( \frac{1}{(2\pi i)^3}\int\limits_{\overline{\mathcal{A}}_2\backslash B_{\varepsilon}(\partial \mathcal{A}_2)}g_{10}\wedge \omega_{6}\wedge\omega_{4}\wedge\omega_{12}\right.\\
&\phantom{:=}\left.-\frac{2}{(2\pi i)^{2}}\int\limits_{\partial B_{\varepsilon}(\partial \mathcal{A}_2)}\left[(g_{6}*g_{4}*g_{12})\wedge \dd^c g_{10}-g_{10}\wedge \dd^c (g_{6}*g_{4}*g_{12})\right]\right)\\[0.2cm]
(B)&:=-\lim_{\varepsilon \rightarrow 0}\left(\frac{2}{2\pi i} \int\limits_{\partial B_\eps(\partial \mathcal{H})}\left[g_4\wedge \dd^c(\sigma_{6,4}g_6)\wedge\omega_{12}-(\sigma_{6,4}g_6)\wedge \dd^c g_4\wedge\omega_{12}\right]\right)\\[0.2cm]
& \phantom{:=}+\frac{1}{(2\pi i)^{2}}\int\limits_{\overline{\mathcal{H}}}g_{6}\wedge\omega_{4}\wedge \omega_{12}+\frac{1}{4\pi i}\int\limits_{\{i\}\times_\mathcal{H} \overline{\Ab}_{1}}g_{4}\wedge\omega_{12}+\frac{1}{6}\int\limits_{\{i\}\times_\mathcal{H}\{\omega\}} g_{12}.
\end{align*}
\end{prop}
\begin{proof}
We start by applying Proposition~\ref{prop_integrals} with
$X=\overline{\mathcal{A}}_2$ and $D=D_1\cup D_2$, where $D_1=\partial \mathcal{A}_2$ and $D_2=\emptyset$.
We consider the cycles 
\begin{align*}
Y=\divi (\chi_{10})=\partial \mathcal{A}_2+\overline{\mathcal{H}},\quad 
Z=\divi (E_6)\cdot \divi (E_4)\cdot \divi (\chi_{12})=\emptyset,
\end{align*}
having codimension $1$ resp. $3$, so $d=3$. By Lemma \ref{lem_properintersection},
$Y$ and $Z$ intersect properly. 
 Clearly, $\vert Y\vert\cap D_2=\emptyset$, and $\vert Z\vert\cap D_1=\emptyset$ holds, as $D_1$ is a component of $Y$. Thus, by Proposition~\ref{prop_integrals}, we obtain the equality
\begin{align*}
\frac{1}{(2\pi i)^3}\int\limits_{\overline{\mathcal{A}}_2} g_{10}*(g_{6}*g_{4}*g_{12})
=(A)+
\frac{1}{(2\pi i)^{2}}\int\limits_{\overline{\mathcal{H}}}g_{6}*g_{4}*g_{12},
\end{align*}
observing that $\overline{Y\backslash Y\cap D_1}=\overline{\mathcal{H}}$.
We will later see that the first integral in the decomposition converges for $\eps$ approaching 0 and can therefore be taken out of the limit.
To treat the second term, we will next apply Proposition~\ref{prop_integrals} with
$X=\overline{\mathcal{H}}$ and $D=D_1\cup D_2$, where $D_1=\emptyset$ and $D_2=\partial \mathcal{H}$.
We consider the cycles 
\begin{align*}
&Y=\divi (E_6)\cdot \divi (E_4)\cdot\overline{\mathcal{H}}=\frac{1}{6}(\{i\}\times_\mathcal{H} \{\omega\}),\quad Z= \divi (\chi_{12})\cdot\overline{\mathcal{H}}= \{i\infty\}\times_\mathcal{H} \overline{\Ab}_1,
\end{align*}
having codimension $2$ and $1$, so $d=2$. Hence,
$Y$ and $Z$ intersect properly, 
and 
 clearly, $\vert Y\vert\cap D_2=\emptyset$ and $\vert Z\vert\cap D_1=\emptyset$. By Proposition~\ref{prop_integrals}, we then obtain
\begin{align*}
\frac{1}{(2\pi i)^{2}}\int\limits_{\overline{\mathcal{H}}}(g_{6}*g_{4})*g_{12}
&=\lim_{\varepsilon \rightarrow 0}\Biggl(\frac{1}{(2\pi i)^{2}}\int\limits_{\overline{\mathcal{H}}\setminus B_\eps(\partial \mathcal{H})}(g_{6}*g_4)\wedge \omega_{12}\Biggr)
+\frac{1}{6}\int\limits_{\{i\}\times_\mathcal{H} \{\omega\}}g_{12},
\end{align*}
observing that $\overline{Y\backslash Y\cap D_1}=\overline{Y}$.
To treat the first term, we now 
rewrite $g_6*g_4$, using \eqref{eq_starprod}, and obtain
\begin{align*}
g_6*g_4&=4\pi i\bigl((\sigma_{4,6} g_6)\dd\dd^c g_4+\dd\dd^c(\sigma_{6,4} g_6)\wedge g_4\bigr)\\
&=4\pi i \bigl( \left(g_6\wedge \dd\dd^c g_4-(\sigma_{6,4}g_6)\wedge \dd\dd^c g_4\right)+\left(\dd(\dd^c(\sigma_{6,4}g_6)\wedge g_4)-\dd^c(\sigma_{6,4} g_6)\wedge\dd g_4\right)\bigr)\\
&=4\pi i\bigl(g_6\wedge\dd\dd^c g_4+\dd\left(\dd^c(\sigma_{6,4}g_6)\wedge g_4-(\sigma_{6,4}g_6)\wedge \dd^c g_4\right)\bigr),
\end{align*}
with $\{\sigma_{4,6},\sigma_{6,4}\}$ a partition of unity adapted to $\divi(E_4)$ and $\divi(E_6)$.
 Note that this expression is singular along $\divi(E_6)\cdot \overline{\mathcal{H}}=\frac{1}{2}(\{i\}\times_\mathcal{H} \overline{\Ab}_1)$. We therefore set $\mathcal{H}_\eps:=\overline{\mathcal{H}}\setminus B_\eps(\partial \mathcal{H}\cup \{i\}\times_\mathcal{H} \overline{\Ab}_1)$ and obtain
\begin{align*}
&\frac{1}{(2\pi i)^2}\int\limits_{\mathcal{H}_\eps}(g_{6}*g_{4})\wedge\omega_{12})=\frac{1}{(2\pi i)^2}\int\limits_{\mathcal{H}_\eps}g_6\wedge\omega_{4}\wedge\omega_{12}\\
&\phantom{aaaa}+\frac{4\pi i}{(2\pi i)^2}\int\limits_{\mathcal{H}_\eps}\dd\bigl(g_4\wedge \dd^c(\sigma_{6,4}g_6)\wedge\omega_{12}-(\sigma_{6,4}g_6)\wedge \dd^c g_4\wedge\omega_{12}\bigr).
\end{align*}
Applying Stokes' theorem and taking the limit $\eps\rightarrow 0$, we obtain
\begin{align*}
&\lim_{\eps \rightarrow 0}\frac{1}{(2\pi i)^2}\int\limits_{\mathcal{H}_\eps}(g_{6}*g_{4})\wedge\omega_{12}=\frac{1}{(2\pi i)^2}\int\limits_{\overline{\mathcal{H}}}g_6\wedge\omega_4\wedge\omega_{12}+\frac{1}{4\pi i}\int\limits_{\{i\}\times_\mathcal{H} \overline{\Ab}_1}\hspace{-0.2cm}g_4\wedge\omega_{12}\\
&\phantom{aaaaaa}-\lim_{\eps \rightarrow 0}\frac{2}{2\pi i}\int\limits_{\partial B_\eps(\partial \mathcal{H})}\left[g_4\wedge \dd^c(\sigma_{6,4}g_6)\wedge\omega_{12}-(\sigma_{6,4}g_6)\wedge \dd^c g_4\wedge\omega_{12}\right].
\end{align*}
This completes the proof of the theorem.
\end{proof}

\section{The Faltings height of the Humbert surface}\label{sect:4}

In this section, we treat the term $(B)$ arising in Proposition \ref{prop_5integrals} of Section \ref{sect:3}.
First, using a well-known result by Rohrlich, we are able to explicitly compute the integrals in the line 
of $(B)$, i.e. over the positive codimensional cycles of $\overline{\mathcal{A}}_2$. Then, we show the vanishing 
of the boundary term, by explicit estimates. We start by recalling Rohrlich's result.

\begin{lem}\label{lem_1dimint} Let $f$ be a modular form of weight $k$ with $f(i\infty)=1$. Then, we have the equality
\begin{align*}
\int\limits_{\Ab_1}\log\Vert f(\tau_1)\Vert_{\Pet} \frac{\dd x_1\dd y_1}{4\pi y^2_1}
=-k\left(\frac{1}{2}\zeta(-1)+\zeta'(-1)\right)-\frac{1}{12}\sum_{\tau_0\in\vert\divi(f)\vert}\mathrm{ord}_{\tau_0}(f)\log\Vert\Delta(\tau_0)\Vert_{\Pet}.
\end{align*}
Here, $\tau_1=x_1+iy_1$ is the coordinate on $\Ab_1$, and the term $\Vert\Delta(\tau_0)\Vert_{\Pet}$ denotes the quantity $(4\pi)^{6}\vert\Delta(\tau_0)y^6_0\vert$.
\end{lem}
\begin{proof} The lemma follows from a theorem of Rohrlich, see \cite{rohrlich}. For the computations leading to the stated version of the formula, see the proof of Theorem~1.6.1 in \cite{kuehn}.
\end{proof}

\begin{prop}\label{prop_lowerdim}The second line in $(B)$, see Proposition~\ref{prop_5integrals},  has the value
\begin{align*}
&\frac{1}{(2\pi i)^{2}}\int\limits_{\overline{\mathcal{H}}}g_{6}\wedge\omega_{4}\wedge \omega_{12}+\frac{1}{4\pi i}\int\limits_{\{i\}\times_\mathcal{H} \overline{\Ab}_{1}}g_{4}\wedge\omega_{12}+\frac{1}{6}\int\limits_{\{i\}\times_\mathcal{H}\{\omega\}} g_{12}\\
&=-6\biggl(\frac{1}{2}+\frac{\zeta'(-1)}{\zeta(-1)}\biggr)-\frac{4}{3}\log 2-\frac{2}{3}\log 3.
\end{align*}
\end{prop}
\begin{proof} We will show the equality
\begin{align*}
\frac{1}{(2\pi i)^{2}}\int\limits_{\overline{\mathcal{H}}}g_{6}\wedge \omega_{12}\wedge\omega_{4}=48\biggl(\frac{1}{2}\zeta(-1)+\zeta'(-1)\biggr)+\frac{1}{3}\log\Vert\Delta(i)\Vert_{\Pet}.\end{align*}
The other terms can be computed in a similar way. 
Noting that $\tau_{12}=0$ on $\mathcal{H}$, we see that the Chern forms $\omega_{12}$ and $\omega_{4}$ reduce to 
\begin{align*}
\omega_{12}=6\cdot4\pi i\left(\frac{\dd x_1\dd y_1}{4\pi y^2_1}+\frac{\dd x_2\dd y_2}{4\pi y^2_2}\right),\quad
\omega_{4}=2\cdot4\pi i\left(\frac{\dd x_1\dd y_1}{4\pi y^2_1}+\frac{\dd x_2\dd y_2}{4\pi y^2_2}\right)
\end{align*}
on $\mathcal{H}$. For their product $\omega_{4}\wedge\omega_{12}$, one obtains
\begin{align*}
\omega_{4}\wedge\omega_{12}=24(4\pi i)^2\frac{\dd x_1\dd y_1}{4\pi y^2_1}\wedge \frac{\dd x_2\dd y_2}{4\pi y^2_2}.
\end{align*}
Recalling the isomorphism $\mathcal{H}\cong\mathrm{Sym}_2(\mathcal{A}_1)$ (insert in Section 3), we note that $\overline{\Ab}_1\times \overline{\Ab}_1$ is a double cover of $\overline{\mathcal{H}}$. We obtain
\begin{align*}
&\frac{1}{(2\pi i)^{2}}\int\limits_{\overline{\mathcal{H}}}g_{6}\wedge \omega_{4}\wedge\omega_{12}=24\frac{(4\pi i)^2}{(2\pi i)^{2}}\int\limits_{\overline{\mathcal{H}}}g_{E_6}\wedge \frac{\dd x_1\dd y_1}{4\pi y^2_1}\wedge \frac{\dd x_2\dd y_2}{4\pi y^2_2}\\
&\phantom{a}=96\int\limits_{\overline{\mathcal{H}}}g_{6}\wedge \frac{\dd x_1\dd y_1}{4\pi y^2_1}\wedge \frac{\dd x_2\dd y_2}{4\pi y^2_2}
=48\int\limits_{\overline{\Ab}_1\times \overline{\Ab}_1}g_{6}\wedge \frac{\dd x_1\dd y_1}{4\pi y^2_1}\wedge \frac{\dd x_2\dd y_2}{4\pi y^2_2}.
\end{align*}
With $g_{6}=-\log\Vert E_6\Vert_{\Pet}$ and the decomposition
 of $E_6$ on $\mathcal{H}$ as $E_6(\tau)=E_6(\tau_1)E_6(\tau_2)$, the latter integral writes as
\begin{align*}
&48\int\limits_{\overline{\Ab}_1\times \overline{\Ab}_1}g_{6}\wedge \frac{\dd x_1\dd y_1}{4\pi y^2_1}\wedge \frac{\dd x_2\dd y_2}{4\pi y^2_2}\\
&\phantom{a}=48\int\limits_{\overline{\Ab}_1\times \overline{\Ab}_1}-\log\Vert E_6(\tau)\Vert_{\Pet}\wedge \frac{\dd x_1\dd y_1}{4\pi y^2_1}\wedge \frac{\dd x_2\dd y_2}{4\pi y^2_2}\\
&\phantom{a}=48\int\limits_{\overline{\Ab}_1}\int\limits_{\overline{\Ab}_1}\left[-\log\Vert E_6(\tau_1)\Vert_{\Pet}-\log\Vert E_6(\tau_2)\Vert_{\Pet}\right] \frac{\dd x_1\dd y_1}{4\pi y^2_1} \frac{\dd x_2\dd y_2}{4\pi y^2_2}.
\end{align*}
With respect to the volume form $\frac{\dd x_1\dd y_1}{4\pi y^2_1}$, the volume of $\mathcal{A}_1$ equals $\frac{1}{12}$. Hence, we can simplify the last integral to
\begin{align}\label{eq_2dint}
&48\int\limits_{\overline{\Ab}_1}\int\limits_{\overline{\Ab}_1}\left[-\log\Vert E_6(\tau_1)\Vert_{\Pet}-\log\Vert E_6(\tau_2)\Vert_{\Pet}\right] \frac{\dd x_1\dd y_1}{4\pi y^2_1} \frac{\dd x_2\dd y_2}{4\pi y^2_2}\notag\\
&\phantom{a}=48\left(\frac{1}{12}\int_{\overline{\Ab}_1}-\log\Vert E_6(\tau_1)\Vert_{\Pet}\frac{\dd x_1\dd y_1}{4\pi y^2_1}+\frac{1}{12}\int\limits_{\overline{\Ab}_1}-\log\Vert E_6(\tau_2)\Vert_{\Pet} \frac{\dd x_2\dd y_2}{4\pi y^2_2}\right) \notag\\
&\phantom{a}=-8\int\limits_{\overline{\Ab}_1}\log\Vert E_6(\tau_1)\Vert_{\Pet} \frac{\dd x_1\dd y_1}{4\pi y^2_1}.
\end{align}

Noting that $\mathrm{ord}_i(E_6)=\frac{1}{2}$, Lemma~\ref{lem_1dimint} yields
 \begin{align}\label{eq_intE6}
\int\limits_{\mathcal{A}_{1}}\log\Vert E_6\Vert_{\Pet} \frac{\dd x\dd y}{4\pi y^2}=-6\biggl(\frac{1}{2}\zeta(-1)+\zeta'(-1)\biggr)-\frac{1}{24}\log\Vert\Delta(i)\Vert_{\Pet}.
\end{align}
Plugging (\ref{eq_intE6}) into (\ref{eq_2dint}), proves the claim of the proposition.
\end{proof}

\begin{prop}\label{prop_boundaryH} The integral in the limit in $(B)$
$$
\frac{4}{2\pi i}\int\limits_{\partial B_\eps(\partial \mathcal{H})}\left[g_4\wedge \dd^c(\sigma_{6,4}g_6)\wedge\omega_{12}-(\sigma_{6,4}g_6)\wedge \dd^c g_4\wedge\omega_{12}\right]
$$
converges absolutely, and its value tends to 0 for $\eps$ approaching 0.
\end{prop}
\begin{proof}
Noting that $\tau_{12}=0$ and $y_1=M$ are constant on $\partial B_\eps(\partial \mathcal{H})$, the form $\omega_{12}$ reduces to
$$
\omega_{12}=4\pi i\dd\dd^c g_{12}=6i\frac{\dd x_2\dd y_2}{y^2_2}
$$
on $\partial B_\eps(\partial \mathcal{H})$. As $\sigma_{6,4}$ depends only on the coordinate $\tau_2$, the form $\ddc\sigma_{6,4}\wedge\omega_{12}$ vanishes. Estimating $\ddc g_4$ and $\ddc g_6$, we can bound the integrand by
$$
\left\vert g_4\wedge \dd^c(\sigma_{6,4}g_6)\wedge\omega_{12}-(\sigma_{6,4}g_6)\wedge \dd^c g_4\wedge\omega_{12}\right\vert\prec (\vert g_4\vert+\vert g_6\vert)\frac{1}{M}\dd\theta\frac{\dd x_2\dd y_2}{y^2_2}.
$$
Furthermore, note that on $\partial B_\eps(\partial \mathcal{H})$, we have the equality $\tau_1=x_1+iM$, with $M=-\log{\eps}/2\pi$. The equality $E_k(\tau)=E_k(\tau_1)E_k(\tau_2)$ on $\mathcal{H}$ ($k=4,6$), see (insert in Section 3), induces the decomposition of the Green forms $g_4(\tau)$ and $g_6(\tau)$ as
$$
g_k(\tau)=g_k(\tau_1)+g_k(\tau_2) \quad (k=4,6)
$$
on $\partial B_\eps(\partial \mathcal{H})$. As the degree 1 Eisenstein series $E_4(\tau_1)$ and $E_6(\tau_1)$ do not vanish for $M$ approaching $\infty$, the term $$g_k(\tau_1)=-\log\Vert E_k(\tau_1)\Vert_{\Pet}=-\vert E_k(\tau_1)\vert-k/2\log(4\pi)-k/2\log M \quad(k=4,6)$$
 is of order $\log M$ for $\tau_1=x_1+iM$. Therefore, on $\partial B_\eps(\partial \mathcal{H})$, we obtain
$$
\vert g_k(\tau)\vert\prec \log M+\vert g_k(\tau_2)\vert \quad (k=4,6).
$$
The domain of integration $\partial B_\eps(\partial \mathcal{H})\subseteq\mathcal{H}$ is of the form $\partial B_\eps(i\infty)\times_{\mathcal{H}}\Ab_1$.

 Hence, we can estimate
\begin{align*}
&\left\vert\frac{4}{2\pi i}\int\limits_{\partial B_\eps(\partial \mathcal{H})}\left[g_4\wedge \dd^c(\sigma_{6,4}g_6)\wedge\omega_{12}-(\sigma_{6,4}g_6)\wedge \dd^c g_4\wedge\omega_{12}\right]\right\vert\\
&\phantom{aaa}\prec \frac{\log M}{M}+ \frac{1}{M}\int\limits_0^{2\pi}\int\limits_{\Ab_1} \left[\vert g_4(\tau_2)\vert+\vert g_6(\tau_2)\vert \right]\frac{\dd x_2\dd y_2}{y^2_2}\dd\theta.
\end{align*}
As the integrals of $g_4(\tau_2)$ and $g_6(\tau_2)$ over $\Ab_1$ converge absolutely, the last sum vanishes for $\eps$ approaching $0$, and, hence, $M$ approaching $\infty$, and the claim follows.
\end{proof}
The two previous theorems imply the following result.
\begin{thm}\label{thm_boundaryterm} The term $(B)$ in Proposition~\ref{prop_5integrals}
converges absolutely, and its value tends to 0 for $\eps$ approaching 0.
\end{thm}

\section{A boundary integral  }\label{sect:5}

In this section, we treat the term $(A)$ arising in Proposition \ref{prop_5integrals} of Section \ref{sect:3}.
First, we are able to trace back one part of $(A)$ to a deep result by Kudla.
The main part of this section is then devoted to proof of the vanishing of remaining boundary integral.

\begin{prop}\label{prop_higherdim} The limit for the first term in $(A)$ in Proposition~\ref{prop_5integrals} 
exists and has the value
\begin{align*}
&\frac{1}{(2\pi i)^3}\int\limits_{\overline{\mathcal{A}}_2}g_{10}\wedge \omega_{6}\wedge\omega_{4}\wedge\omega_{12}\\
&\phantom{aaaaaaaaa}=10\cdot 6\cdot 4\cdot 12\,\zeta(-3)\zeta(-1)\left(\frac{4}{3}+2\frac{\zeta'(-3)}{\zeta(-3)}-\frac{\zeta'(-1)}{\zeta(-1)}+\frac{6}{5}\log 2\right).
\end{align*}
\end{prop}
\begin{proof}
This follows from a result of Kudla, see \cite{kudla}.
\end{proof}
In the following, we will prove the vanishing of term $(B)$ in Proposition~\ref{prop_5integrals}. The proof of the vanishing of the first integral in $(B)$ requires several steps, simplifying the integrand and dividing the domain of integration. We will prepare the proof with some preliminary lemmata.

\begin{lem}\label{lem_integrand1}
On $\partial B_\eps (\partial \Ab_2)$, we have the equality
\begin{align*}
&\frac{1}{(4\pi i)^2}(g_{6}*g_{4}*g_{12})\wedge \dd^c g_{10}\\
&\phantom{a}=(\sigma_{4,6}g_{6})\wedge\dd\ddc g_4\wedge\dd\ddc g_{12}\wedge\dd^c g_{10}+\dd\ddc (\sigma_{6,4}g_6)\wedge g_4\wedge \dd\ddc g_{12}\wedge \ddc g_{10}.
\end{align*}
\end{lem}

\begin{proof}
To compute the form $g_{6}*g_{4}*g_{12}$, we first consider the non-intersecting cycles $\divi(E_6)\cdot\divi(E_4)$ and $\divi(\chi_{12})=\partial\mathcal{A}_2$. Let now $\{\sigma_{(6,4), 12},\sigma_{12,(6,4)}\}$ be a partition of unity adapted to $\divi(E_6)\cdot\divi(E_4)$ and $\divi(\chi_{12})$. By the definition of the $*$-product \eqref{eq_starprod}, we obtain 
\begin{align*}
&g_{6}*g_{4}*g_{12}=4\pi i\bigl((\sigma_{12,(6,4)}(g_{6}*g_{4}))\wedge\dd\dd^cg_{12}+\dd\dd^c(\sigma_{(6,4), 12}(g_{6}*g_{4}))\wedge g_{12}\bigr).
\end{align*}
In a neighbourhood of $\divi(\chi_{12})=\partial\Ab_2$, we have $\sigma_{(6,4), 12}=0$ and $\sigma_{12,(6,4)}=1$. For $\eps$ small, we therefore obtain on $\partial B_{\eps}(\partial\Ab_2)$ the equality
\begin{align*}
g_{6}*g_{4}*g_{12}=4\pi i(g_{6}*g_{4})\wedge\dd\dd^cg_{12}.
\end{align*} 
Using the definition of the $*$-product to compute $g_{6}*g_{4}$,  we then obtain on $\partial B_{\eps}(\partial\Ab_2)$ the equality
\begin{align*}
&g_{6}*g_{4}*g_{12}=(4\pi i)^2\bigl((\sigma_{4,6}g_{6})\wedge\dd\ddc g_4\wedge\dd\ddc g_{12}+\dd\ddc (\sigma_{6,4}g_6)\wedge g_4\wedge \dd\ddc g_{12}\bigr),
\end{align*}
with $\{\sigma_{4,6},\sigma_{6,4}\}$ a partition of unity  adapted to $\divi(E_4)$ and $\divi(E_{6})$. This proves the claim of the lemma.
\end{proof}

\begin{lem}\label{lem_integrand2}
On $\partial B_\eps (\partial \Ab_2)$, we have the equality
\begin{align*}
&\frac{1}{(4\pi i)^2}g_{10}\wedge \dd^c(g_{6}*g_{4}*g_{12})\\
&\phantom{a}=g_{10}\wedge \ddc(\sigma_{4,6}g_{6})\wedge\dd\ddc g_4\wedge\dd\ddc g_{12}+g_{10}\wedge\dd\ddc (\sigma_{6,4}g_6)\wedge \ddc g_4\wedge \dd\ddc g_{12}.
\end{align*}
\end{lem}

\begin{proof}
The claim follows by a straight-forward computation, as in the proof of Lemma \ref{lem_integrand1}.
\end{proof}

To give estimates for the integrand $(g_{6}*g_{4}*g_{12})\wedge \dd^c g_{10}-g_{10}\wedge \dd^c(g_{6}*g_{4}*g_{12})$, we will embed the domain of integration $\partial B_\eps(\partial\Ab_2)$ into a fundamental domain for $\Gamma_2\backslash\MH_2$.
\begin{lem}\label{lem_fundamentalbound}
Let $M=-\log{\eps}/2\pi$. There is a fundamental domain $\mathcal{F}\subseteq\MH_2$ for the action of $\Gamma_2$ on $\MH_2$ such that the preimage of $\partial B_\eps(\partial\Ab_2)$ under the quotient morphism $\pi_2\colon\MH_2\longrightarrow \Ab_2$ restricted to $\mathcal{F}$ is contained in the set $S_{\eps}$ given by restricting the local coordinates under consideration as follows:
\begin{align}\label{eq_seps}
y_1=M,\; y_2\in[1/2, M],\; y_{12}\in [0,y_2/2], \;\theta\in [0,2\pi), \; x_2, x_{12}\in[-1/2, 1/2].
\end{align}
\end{lem}
\begin{proof}
Recall that a fundamental domain $\mathcal{F}_2$ for the action of $\Gamma_2$ on $\MH_2$ is given by the following Minkowski conditions on $\tau=x+iy\in \MH_2$:
\begin{itemize}
\item[(i)] For all $\gamma=\left(\begin{smallmatrix}A&B\\C&D\end{smallmatrix}\right)\in\Gamma_2$ the inequality $\vert\det(C\tau+D)\vert\geq 1$ holds.
\item[(ii)] The matrix $y$ is Minkowski reduced, i.e., for all $l\in\MZ^2$ such that the last $g-k+1=3-k$ entries are relatively prime, we have $y_{k}\leq l^tyl$ $(k=1,2)$; furthermore, $y_{12}\geq 0$ holds.
\item[(iii)] The matrix $x$ satisfies $\vert x_{k}\vert\leq\frac{1}{2}$ $(k=1,2, 12)$.
\end{itemize}
Applying condition (ii) with $k=1$ and $l=\left(\begin{smallmatrix}0\\1\end{smallmatrix}\right)$, one obtains $y_{1}\leq y_{2}$. For $k=2$ and 
$l= \left(\begin{smallmatrix}1\\-1\end{smallmatrix}\right)$, one obtains
$$
y_2\leq \begin{pmatrix}1\\-1\end{pmatrix}^t\begin{pmatrix}y_1&y_{12}\\y_{12}&y_2\end{pmatrix}\begin{pmatrix}1\\-1\end{pmatrix}=y_1+y_2-2y_{12},
$$
and the condition $2y_{12}\leq y_1$ follows. We immediately deduce the inequality $\det y\leq y_{1}y_{2}\leq 2\det y$. Furthermore, applying condition (i) for $C=\left(\begin{smallmatrix}1&0\\0&0\end{smallmatrix}\right)$ and $D=\left(\begin{smallmatrix}0&0\\0&1\end{smallmatrix}\right)$, together with condition (iii), one gets $\tau_{1}\in \mathcal{F}_1$, where $\mathcal{F}_1$ is the standard fundamental domain for the action of $\Gamma_1$ on $\MH_1$. Hence, we can assume $y_{1}\geq\sqrt{3}/2>1/2$. Interchanging the roles of $\tau_1$ and $\tau_2$ by translating $\mathcal{F}_2$ employing the action of the matrix 
$$
S:=\begin{pmatrix}0&1&0&0\\ 1&0&0&0\\0&0&0&1\\0&0&1&0
\end{pmatrix}\in \Gamma_2.
$$
on $\MH_2$, we obtain a new fundamental domain $\mathcal{F}=S\mathcal{F}_2$ such that $\partial B_\eps(\partial\Ab_2)$ lies in the set $S_\eps$ as claimed. 
\end{proof}
\begin{rem}
From now on, we will identify $\partial B_\eps(\partial\Ab_2)$ with its preimage in the fundamental domain $\mathcal{F}$ introduced in Lemma~\ref{lem_fundamentalbound}.
\end{rem}
\begin{lem}\label{lem_newpartition}
The partition of unity $\{\sigma_{4,6}, \sigma_{6,4}\}$ adapted to the divisors $\divi(E_4)$ and $\divi(E_6)$ can be chosen in a way such that in a small neighbourhood of $\partial\Ab_2$, and hence on $\partial B_\eps(\partial \Ab_2)$ (for small $\eps>0$), the following properties hold:
\begin{itemize}
\item[(i)] The value of $\sigma_{4,6}$ and $\sigma_{6,4}$ depends only on the value of the coordinate $\tau_2$, i.e., $\sigma_{4,6}(\tau)=\sigma_{4,6}(\tau_2)$ and $\sigma_{6,4}(\tau)=\sigma_{6,4}(\tau_2)$.
\item[(ii)] We have $\sigma_{4,6}(\tau_2)=\sigma_{6,4}(\tau_2)=1/2$ for $y_2>2$.
\end{itemize}
\end{lem}
\begin{proof} By (to be inserted), the value of $E_4$ and $E_6$ restricted to $\partial\mathcal{A}_2$ only depends on the coordinate $\tau_2$. Hence, we can assume the same for the partition of unity in a small neighbourhood of $\partial\Ab_2$. Furthermore, we note that the cycles $\divi(E_4)\cdot\partial\Ab_2$ and $\divi(E_6)\cdot\partial\Ab_2$ are supported in the open set defined by the condition $y_2<2$, as $\mathrm{Im}(i), \mathrm{Im}(\omega)<2$. Therefore, we can choose $\sigma_{4,6}(\tau_2)$ and $\sigma(\tau_2)$ to equal $1/2$ outside this range.
\end{proof}
\begin{lem}\label{lem_thetaeps} For any $\vartheta$-series $\vartheta_{a,b}$ as in (\ref{eq_thetaseries}), we have the bound 
$$\left\vert\frac{\partial}{\partial x_1}\vartheta_{a,b}(\tau)\right\vert=2\pi\left\vert\frac{\partial}{\partial \theta}\vartheta_{a,b}(\tau)\right\vert\prec\eps^{\frac{1}{16}}$$
for $\tau\in S_\eps$ defined in Lemma~\ref{lem_fundamentalbound}. 
\end{lem}
\begin{proof}
On $\partial\Ab_2$, the coordinate $\tau_1$ is constant, equal to $i\infty$. Hence, the partial derivative $\frac{\partial}{\partial x_1}\vartheta_{a,b}$ vanishes on $\partial\Ab_2$. To determine its vanishing order, we note that for a matrix $\tau\in S_\eps$ the inequalities $y_{12}\leq y_2/2$ and $y_2<y_1$ hold by definition of $S_\eps$. We deduce the inequality $(y_1/2)y_2-y^2_{12}>0$. Therefore, the matrix 
$$\tau'=\begin{pmatrix}\tau_1-i\frac{y_1}{2}&\tau_{12}\\ \tau_{12}&\tau_2\end{pmatrix},
$$ obtained from $\tau$ by replacing the coordinate $y_1$ by $y_1/2 $, lies in $\MH_2$.
Letting $n=(n_1,n_2)^t\in\MZ^2$ and $a=(a_1, a_2)^t, b=(b_1, b_2)^t\in (\MZ/2\MZ)^2$, we expand $\vartheta_{a,b}$ as
\begin{align*}
&\vartheta_{a,b}(\tau)=\sum_{n\in \MZ^2} e^{2\pi i\Bigl(\frac{1}{2}\bigl(n+\frac{a}{2}\bigr)^t\tau\bigl(n+\frac{a}{2}\bigr)+\bigl(n+\frac{a}{2}\bigr)^t\frac{b}{2}\Bigr)}\\
&=\sum_{n\in \MZ^2} e^{\pi i\Bigl(\bigl(n_1+\frac{a_1}{2}\bigr)^2\tau_1+2\bigl(n_1+\frac{a_1}{2}\bigr)\bigl(n_2+\frac{a_2}{2}\bigr)\tau_{12}+\bigl(n_2+\frac{a_2}{2}\bigr)^2\tau_2+\bigl(n_1+\frac{a_1}{2}\bigr)b_1+\bigl(n_2+\frac{a_2}{2}\bigr)b_2\Bigr)}\\
&=\sum_{n\in \MZ^2} e^{-\pi\bigl(n_1+\frac{a_1}{2}\bigr)^2 \frac{y_1}{2}} e^{2\pi i\Bigl(\frac{1}{2}\bigl(n+\frac{a}{2}\bigr)^t\tau'\bigl(n+\frac{a}{2}\bigr)+\bigl(n+\frac{a}{2}\bigr)^t\frac{b}{2}\Bigr)}.
\end{align*}
For $n_1\neq0$ or $a_1\neq 0$, we deduce with $y_1=-\log \eps/2\pi$ that
\begin{align*}
e^{-\pi\bigl(n_1+\frac{a_1}{2}\bigr)^2 \frac{y_1}{2}}\leq e^{\frac{1}{16}\log\eps}=\eps^{\frac{1}{16}}.
\end{align*}
For the partial derivative $\frac{\partial}{\partial x_1}\vartheta_{a,b}$, we now obtain
\begin{align*}
\left\vert\frac{\partial}{\partial x_1}\vartheta_{a,b}(\tau)\right\vert
&=\left\vert\sum_{n\in \MZ^2} e^{-\pi\left(n_1+\frac{a_1}{2}\right)^2 \frac{y_1}{2}} \pi i\left(n_1+\frac{a_1}{2}\right)^2e^{2\pi i\Bigl(\frac{1}{2}\bigl(n+\frac{a}{2}\bigr)^t\tau'\bigl(n+\frac{a}{2}\bigr)+\bigl(n+\frac{a}{2}\bigr)^t\frac{b}{2}\Bigr)}\right\vert\\
&\leq \eps^{\frac{1}{16}}\sum_{n\in \MZ^2} \left\vert\pi i\left(n_1+\frac{a_1}{2}\right)^2e^{2\pi i\Bigl(\frac{1}{2}\bigl(n+\frac{a}{2}\bigr)^t\tau'\bigl(n+\frac{a}{2}\bigr)+\bigl(n+\frac{a}{2}\bigr)^t\frac{b}{2}\Bigr)}\right\vert,
\end{align*}
and the latter sum converges, as the sum
\begin{align*}
\sum_{n\in \MZ^2} \pi i\left(n_1+\frac{a_1}{2}\right)^2e^{2\pi i\Bigl(\frac{1}{2}\bigl(n+\frac{a}{2}\bigr)^t\tau'\bigl(n+\frac{a}{2}\bigr)+\bigl(n+\frac{a}{2}\bigr)^t\frac{b}{2}\Bigr)}=\frac{\partial}{\partial x_1}\vartheta_{a,b}(\tau')
\end{align*}
converges absolutely for $\tau'\in\MH_2$. This proves the claim.
\end{proof}
We will now prove the vanishing of the integral along $\partial B_\eps(\partial \Ab_2)$ by suitably subdividing the integration domain.

\begin{thm}\label{thm_fullboundary} The integral
$$
\frac{1}{(4\pi i)^2}\int\limits_{\partial B_\eps(\partial \Ab_2)} \left[(g_{6}*g_{4}*g_{12})\wedge \dd^c g_{10}-g_{10}\wedge \dd^c (g_{6}*g_{4}*g_{12})\right]
$$
converges absolutely, and its value tends to 0 for $\eps$ approaching 0.
\end{thm}

\begin{proof}
 Let $\{\sigma_{4,6}, \sigma_{6,4}\}$ be a partition of unity adapted to the divisors $\divi(E_4)$ and $\divi(E_6)$. Let $N_4\subseteq\partial B_\eps (\partial \Ab_2)$ be a neighbourhood of $\vert\divi(E_4)\vert\cap \partial B_\eps (\partial \Ab_2)$ such that $\sigma_{4,6}=1$ on $N_4$, and let $N_6\subseteq\partial B_\eps (\partial \Ab_2)$ be a neighbourhood of $\vert\divi(E_6)\vert\cap \partial B_\eps (\partial \Ab_2)$ such that $\sigma_{6,4}=1$ on $N_6$.
We define an open subset $U$ of $\partial B_\eps(\partial \Ab_2)$ by setting
\begin{align*}
U:=N_4\cup N_6\cup (\partial B_\eps(\partial \Ab_2)\cap\{\tau\in\mathcal{F}\,\vert\, y_2>2\})\subseteq\partial B_\eps(\partial \Ab_2).
\end{align*}
 We note that this is a disjoint union. We first show that for $U\subseteq\partial B_\eps(\partial \Ab_2)$ as before, the integral
\begin{align*}
\frac{1}{(4\pi i)^2}\int\limits_{U} \left[(g_{6}*g_{4}*g_{12})\wedge \dd^c g_{10}-g_{10}\wedge \dd^c (g_{6}*g_{4}*g_{12})\right]
\end{align*}
converges absolutely, and its value tends to 0 for $\eps$ approaching 0. Applying Lemmas~\ref{lem_integrand1} and \ref{lem_integrand2}, we see that the form 
\begin{align*}
\frac{1}{(4\pi i)^2}\bigl((g_{6}*g_{4}*g_{12})\wedge \dd^c g_{10}-g_{10}\wedge \dd^c(g_{6}*g_{4}*g_{12})\bigr)
\end{align*}
restricts on $N_4$ to
\begin{align}\label{eq_n4integrand}
g_{6}\wedge\dd\ddc g_4\wedge\dd\ddc g_{12}\wedge\dd^c g_{10}-g_{10}\wedge \ddc g_{6}\wedge\dd\ddc g_4\wedge\dd\ddc g_{12},
\end{align}
on $N_6$ to
\begin{align}\label{eq_n6integrand}
\dd\ddc g_6\wedge g_4\wedge \dd\ddc g_{12}\wedge \ddc g_{10}-g_{10}\wedge\dd\ddc g_6\wedge \ddc g_4\wedge \dd\ddc g_{12},
\end{align}
and on $\partial B_\eps(\partial \Ab_2)\cap\{\tau\in\mathcal{F}\,\vert\, y_2>2\}$ to
\begin{align}\label{eq_nintegrand}
\frac{1}{2}&\bigl(g_{6}\wedge \dd\dd^c g_{4}\wedge\dd\dd^c g_{12}\wedge \dd^c g_{10}-g_{10}\wedge \dd^c g_{6}\wedge \dd\dd^c g_{4}\wedge\dd\dd^c g_{12}\bigr.\notag\\
&\bigl.+ \dd\dd^c g_{6}\wedge g_{4}\wedge\dd\dd^c g_{12}\wedge \dd^c g_{10}-g_{10}\wedge \dd\dd^c g_{6}\wedge \dd^c g_{4}\wedge\dd\dd^c g_{12}\bigr),
\end{align}
as here $\sigma_{4,6}=\sigma_{6,4}=1/2$ holds.
With the explicit formulas for $g_4, g_6, g_{10}, g_{12}$ (to be inserted in Section 2), we can explicitely express the integrand on $U$ in the form
\begin{align*}
& f(\tau)\dd\theta\dd x_2\dd y_2\dd x_{12}\dd y_{12},
\end{align*}
where $f(\tau)$ is a smooth function depending on $\eps$. 
We will first bound $ f(\tau)$ on $N_4$, and then deduce that the same bound holds for the whole $U$.
With $M=-\log \eps/2\pi$, one finds
\begin{align}\label{eq_fexplicit}
f(\tau)=&\frac{3}{16\pi^3}\left(-\log\vert E_6\vert-6\log(4\pi)-3\log\left(My_2-y_{12}^2\right)\right)\times\notag\\
&\Biggl(\frac{4M^2}{\left(My_2-y_{12}^2\right)^3}\biggl(-\eps\frac{1}{\vert \chi_{10}\vert}\frac{\partial}{\partial r}\vert \chi_{10}\vert +\frac{5}{2\pi}\frac{y_2}{My_2-y_{12}^2}\biggr)\notag\\
&\phantom{a}+\frac{4 y^2_{12}}{\left(My_2-y_{12}^2\right)^3}\frac{1}{2\pi}\biggl(\frac{1}{\vert \chi_{10}\vert }\frac{\partial}{\partial y_2}\vert \chi_{10}\vert +5\frac{M}{My_2-y_{12}^2} \biggr)\notag\\
&\phantom{a}+\frac{4My_{12}}{\left(My_2-y_{12}^2\right)^3}\frac{1}{2\pi}\biggl(\frac{1}{\vert \chi_{10}\vert }\frac{\partial}{\partial y_{12}}\vert \chi_{10}\vert -10\frac{y_{12}}{My_2-y_{12}^2}\biggr)
\Biggr)\notag\\
&-\frac{3}{16\pi^3}\biggl(-\log\vert \chi_{10}\vert-10\log(4\pi)-5\log\left(My_2-y_{12}^2\right)\biggr)\times\notag\\
&\Biggl(\frac{4M^2}{\left(My_2-y_{12}^2\right)^3}\biggl(-\eps\frac{1}{\vert E_6\vert}\frac{\partial}{\partial r}\vert E_6\vert +\frac{3}{2\pi}\frac{y_{2}}{My_2-y_{12}^2}\biggr)\notag\\
&\phantom{a}+\frac{4 y^2_{12}}{\left(My_2-y_{12}^2\right)^3}\frac{1}{2\pi}\biggl(\frac{1}{\vert E_6\vert }\frac{\partial}{\partial y_2}\vert E_6\vert +3\frac{M}{My_2-y_{12}^2} \biggr)\notag\\
&\phantom{a}+\frac{4My_{12}}{\left(My_2-y_{12}^2\right)^3}\frac{1}{2\pi}\biggl(\frac{1}{\vert E_6\vert }\frac{\partial}{\partial y_{12}}\vert E_6\vert -6\frac{y_{12}}{My_2-y_{12}^2} \biggr)
\Biggr).
\end{align}
The modular form $\chi_{10}$ has a simple zero along $\partial\Ab_2$, and a simple zero along $\mathcal{H}$, given by the equality $\tau^2_{12}=0$ for the local coordinate $\tau^2_{12}$ around $\mathcal{H}$. Therefore, it decomposes as $\chi_{10}=r\tau^2_{12}\phi_{10}(\tau)$, with $\phi_{10}(\tau)$ a non-zero smooth function in a neighbourhood of  $\partial\mathcal A_2$. Hence, the terms
$$
\eps\frac{1}{\vert \chi_{10}\vert}\frac{\partial}{\partial r}\vert \chi_{10}\vert,\quad \frac{1}{\vert \chi_{10}\vert }\frac{\partial}{\partial y_2}\vert \chi_{10}\vert, \quad \frac{y_{12}}{\vert \chi_{10}\vert}\frac{\partial}{\partial y_{12}}\vert \chi_{10}\vert
$$
are bounded from above on $\partial B_{\eps}(\partial\mathcal A_2)$, and the bound is independent of $\eps$. Note that, as we are outside $N_6$, the term $\log\vert E_6\vert$ and its partial derivatives occuring in (\ref{eq_fexplicit}) are bounded independently of $\eps$ as well. 
Applying these considerations and noting that the term $\log\vert\chi_{10}\vert$ is of order $M$, we can bound the absolute value of $f(\tau)$ on $N_4$ by
\begin{align*}
\vert f(\tau)\vert\prec\frac{\log\left(My_2-y_{12}^2\right)}{\left(My_2-y_{12}^2\right)^3}\cdot&\Biggl(M^2\biggl(1+\frac{y_2}{My_2-y_{12}^2}\biggr)\\
&+ y^2_{12}\biggl(1+\frac{M}{My_2-y_{12}^2}\biggr)+M\biggl(1+\frac{y^2_{12}}{My_2-y_{12}^2}\biggr)
\Biggr)\\
+\frac{M}{\left(My_2-y_{12}^2\right)^3}\cdot&\Biggl(M^2\biggl( \eps+\frac{y_2}{My_2-y_{12}^2}\biggr)\\
&+ y^2_{12}\biggl(1+\frac{M}{My_2-y_{12}^2}\biggr)+My_{12}\biggl(1+\frac{y_{12}}{My_2-y_{12}^2} \biggr)
\Biggr).
\end{align*}
 With Lemma~\ref{lem_fundamentalbound}, one obtains the estimates
\[
\frac{1}{2}\leq y_2\leq M, \quad 0\leq y_{12}\leq y_2, \quad \frac{1}{My_2}\leq\frac{1}{My_2-y^2_{12}}\leq\frac{2}{My_2}.
\]
Moreover, $\eps\leq 1/M$ for small $\eps>0$. Hence, we can further bound $\vert f(\tau)\vert$ as
\begin{align}\label{eq_bound}
\vert f(\tau)\vert&\prec\frac{\log M}{(My_2)^3}\cdot\Biggl(M^2\biggl(1+\frac{y_2}{My_2}\biggr)+ y^2_{2}\biggl(1+\frac{M}{My_2}\biggr)+M\biggl(1+\frac{y^2_{2}}{My_2}\biggr)
\Biggr) \notag\\
&\phantom{a}+\frac{M}{\left(My_2\right)^3}\cdot\Biggl(M^2\biggl( \eps+\frac{y_2}{My_2}\biggr)+ y^2_{2}\biggl(1+\frac{M}{My_2}\biggr)+My_{2}\biggl(1+\frac{y_{2}}{My_2} \biggr)
\Biggr) \notag\\
&\prec\log M\Bigl(\frac{1}{M y^3_2}+\frac{1}{M^3 y_2}+\frac{1}{M^2 y^3_2}\Bigr) +\frac{1}{M y^3_2}+\frac{1}{M^2 y_2}+\frac{1}{M y^2_2}\notag \\
&\prec \frac{\log M}{M y^3_2}+\frac{1}{My^2_2}.
\end{align}
By analogous computations, one easily sees that the same bound for $\vert f(\tau)\vert$ holds on the whole of $U$.
As, again by Lemma~\ref{lem_fundamentalbound}, the domain of integration is contained in the set $S_{\eps}$ given by the restrictions
$$
y_1=M,\; y_2\in[1/2, M],\; y_{12}\in [0,y_2/2], \;\theta\in [0,2\pi), \; x_2, x_{12}\in[-1/2, 1/2],
$$
the value of the integral of $\left[(g_{6}*g_{4}*g_{12})\wedge \dd^c g_{10}-g_{10}\wedge \dd^c (g_{6}*g_{4}*g_{12})\right]$ over the open set $U$ can be bounded from above by the integral of the estimate for $\vert f(\tau)\vert\dd\theta\dd x_2\dd y_2\dd x_{12}\dd y_{12}$ given in (\ref{eq_bound}) over the set $S_{\eps}$. One obtains
\begin{align*}
&\left\vert\frac{1}{(4\pi i)^2}\int\limits_{U}\left[(g_{6}*g_{4}*g_{12})\wedge \dd^c g_{10}-g_{10}\wedge \dd^c (g_{6}*g_{4}*g_{12})\right]\right\vert\\
&\phantom{aaa}\prec\int\limits_{S_\eps}\left(\frac{\log M}{M y^3_2}+\frac{1}{My^2_2}\right)\dd\theta\dd x_2\dd y_2\dd x_{12}\dd y_{12}\\
&\phantom{aaa}\prec \int\limits_{1/2}^M\left(\frac{\log M}{M y^
2_2}+\frac{1}{My_2}\right)\dd y_2\prec\frac{\log M}{M},
\end{align*}

as integrating over $\theta$ and $x_j$ $(j=1,12,2)$ gives a factor $2\pi$, and integration over $y_{12}$ multiplies the integrand by $y_2/2$. As $\log M/M$ tends to 0 for $\eps$ approaching 0, the claim follows.

We will now show that for $U\subseteq\partial B_\eps(\partial \Ab_2)$ as before, the integral
$$
\frac{1}{(4\pi i)^2}\int\limits_{\partial B_\eps(\partial\Ab_2)\setminus U} \left[(g_{6}*g_{4}*g_{12})\wedge \dd^c g_{10}-g_{10}\wedge \dd^c (g_{6}*g_{4}*g_{12})\right]
$$
converges absolutely, and its value tends to 0 for $\eps$ approaching 0. By Lemmas~\ref{lem_integrand1} and \ref{lem_integrand2}, the integrand has the form
\begin{align}\label{eq_fullintegrand}
&(\sigma_{4,6}g_{6})\wedge\dd\ddc g_4\wedge\dd\ddc g_{12}\wedge\dd^c g_{10}+\dd\ddc (\sigma_{6,4}g_6)\wedge g_4\wedge \dd\ddc g_{12}\wedge \ddc g_{10}\notag\\
&-g_{10}\wedge \ddc(\sigma_{4,6}g_{6})\wedge\dd\ddc g_4\wedge\dd\ddc g_{12}-g_{10}\wedge\dd\ddc (\sigma_{6,4}g_6)\wedge \ddc g_4\wedge \dd\ddc g_{12}
\end{align}
on $\dell B_{\eps}(\dell \mathcal{A}_2)\setminus U$, with $\{\sigma_{4,6}, \sigma_{6,4}\}$ a partition of unity adapted to $\divi (E_4)$ and $ \divi (E_6)$. In the following, we will give bounds for the forms occuring in (\ref{eq_fullintegrand}). For forms $\alpha, \beta$ on $\partial B_\eps(\partial \Ab_2)$, we will use the notation $\alpha\prec \beta$ if there exists a positive real constant $C$ such that $\int_V\alpha\leq C\int_V\beta$ for all closed subsets $V\subseteq\partial B_\eps(\partial \Ab_2)$ whenever the integrals are defined. We compute 
\begin{align*}
\dd g_6
=&-\frac{1}{\vert E_6\vert}\frac{\partial}{\partial \theta}\vert E_6\vert \dd \theta-\frac{1}{\vert E_6\vert }\frac{\partial}{\partial x_2}\vert E_6\vert\dd x_2-\frac{1}{\vert E_6\vert }\frac{\partial}{\partial x_{12}}\vert E_6\vert\dd x_{12}\\
&-\Biggl(\frac{1}{\vert E_6\vert }\frac{\partial}{\partial y_2}\vert E_6\vert +3\frac{M}{My_2-y_{12}^2}\Biggr)\dd y_2\\
&-\Biggl(\frac{1}{\vert E_6\vert }\frac{\partial}{\partial y_{12}}\vert E_6\vert -3\frac{2y_{12}}{My_2-y_{12}^2}\Biggr)\dd y_{12}.
\end{align*}
We can apply Lemma~\ref{lem_thetaeps} and obtain the bound
$$
\left\vert\frac{\partial}{\partial x_1}E_6\right\vert=\frac{1}{2\pi}\left\vert\frac{\partial}{\partial \theta}E_6\right\vert\prec\eps^{\frac{1}{16}}\prec \frac{1}{M}
$$
on $S_\eps$ for $\eps$ small. 

Applying the conditions $1/2<y_2<2$ and $y_{12}\leq y_2$ to the coordinate expansions of $\dd^c g_k$ and $\dd\dd^c g_k$ $(k=4,6,12)$, and noting that $\sigma_{6,4}$ only depends on the local coordinate $\tau_2$, we obtain the estimates
\begin{align*}
\dd \sigma_{6,4}, \dd^c \sigma_{6,4}&\prec \dd x_2+\dd y_2,\\[0.15cm]
\dd\dd^c \sigma_{6,4} &\prec \dd x_2\dd y_2,\\
\dd^c g_6, \dd g_6, \dd^c g_4 &\prec \frac{1}{M} \dd\theta+\dd x_2+\dd y_2+\dd x_{12} +\dd y_{12},\\
\dd^c g_{10} &\prec \dd\theta+\dd x_2+\dd y_2+\dd x_{12} +\dd y_{12}
\end{align*}
and
\begin{align*}
\dd\dd^c g_{4}, \dd\dd^c g_6,\dd\dd^c g_{12}
\prec \frac{1}{M^2}(\dd\theta\dd y_{2}+\dd\theta \dd y_{12})+\frac{1}{M}(\dd x_2\dd y_{12}+\dd x_{12}\dd y_2+\dd x_{12}\dd y_{12})+\dd x_2\dd y_2
\end{align*}
on $\partial B_\eps(\partial\Ab_2)\setminus U$. 

Using these estimates to bound the summands in (\ref{eq_fullintegrand}) on $\partial B_\eps(\partial\Ab_2)\setminus U$, one obtains the bound $\frac{\log(M)^2}{M}$ for the integrand and the theorem follows.
\end{proof}
 
\section{An explicit formula for the arithmetic volume}\label{sect:6}
\begin{prop}
The contribution $\widehat{\vol}(\overline\Ab_2)_{\mathrm{fin}}$ for the arithmetic self intersection number coming from the finite fibres is a rational linear combination of $\log 2 $ and $\log 3$.
\end{prop} 
\begin{proof}The statement can be verified by considering $\vartheta$-embeddings for $\Ab_2(2)$ into projective space, see, e.g., \cite{igusaMFZ}, for the defining equations for its image in $\mathbbm{P}^{4}$. These are only defined over $\MZ[\frac{1}{2}]$.
Setting
\begin{align*}
&t=-\vartheta^2_{0110}&x_{11}=\vartheta^2_{1001}\\
&x_{12}=-i\vartheta^2_{1000}&x_{13}=\vartheta^2_{0100}\\
&x_{21}=i\vartheta^2_{0001}&x_{22}=\vartheta^2_{0000}\\
&x_{23}=i\vartheta^2_{1100}&x_{31}=\vartheta^2_{0010}\\
&x_{32}=-i\vartheta^2_{0011}&x_{33}=\vartheta_{1111},
\end{align*}

the map $\tau \mapsto (t^2(\tau):x^2_{11}(\tau):\ldots:x^2_{33}(\tau))$ gives an embedding of $\mathcal{A}_2$ into $\mathbbm{P}^9$.
The $\vartheta$-relations translate to
\begin{align*}
\sum_{k=1,2,3} x_{j_1k}x_{j_2k}-\delta_{j_1j_2}t^2, \quad \sum_{j=1,2,3} x_{jk_1}x_{jk_2}-\delta_{k_1k_2}t^2
\end{align*}
for $j_1, j_2, j_3, k_1, k_2, k_3\in\{1,2,3\}$.
With these relations, one finds dependencies of the 10 coordinates above and can replace them by 
\begin{align*}
y_0=\vartheta^4_{0110},\;
y_1=\vartheta^4_{0100},\;
y_2=\vartheta^4_{0000},\;
y_3=-\vartheta^4_{1000}-\vartheta^4_{0110},\;
y_4=-\vartheta^4_{1100}-\vartheta^4_{0110},
\end{align*}
so 
we see that $\mathrm{proj}(\mathcal{A}_2)$ is a quartic hypersurface in $\mathbbm{P}^4$, defined by the equation
\begin{align}
(y_0y_1+y_0y_2+y_1y_2-y_3y_4)^2-4y_0y_1y_2(y_0+y_1+y_2+y_3+y_4)=0.\label{eq:projCoord}
\end{align}
In these coordinates, the modular forms become
\begin{align}
&\chi^2_{10}=y_0y_1y_2(-y_2-y_4)(y_0+y_1+y_2+y_3+y_4)(-y_2-y_3)(y_0+y_3)(-y_1-y_3)\notag\\
&\phantom{aaaaaa}\cdot(y_0+y_4)(-y_1-y_4)\label{eq:projChi10}\\
&E_4=y^2_0+y^2_1+y^2_2+(-y_2-y_4)^2+(y_0+y_1+y_2+y_3+y_4)^2+(-y_2-y_3)^2\notag\\
&\phantom{aaaaaa}+(y_0+y_3)^2+(-y_1-y_3)^2+(y_0+y_4)^2+(-y_1-y_4)^2\label{eq:projE4}\\
&E_{6}=(-y_1-y_4)y_2(-y_2-y_3)+(-y_1-y_4)y_2y_0+\ldots\label{eq:projE6}\\
&\chi_{12}=y_0y_1(y_0+y_3)(-y_1-y_3)(y_0+y_4)(-y_1-y_4))+\ldots\label{eq:projChi12}
\end{align}
Note that these equations are symmetric in $y_0, y_1, y_2$ and $y_3,y_4$. One of the factors in equation (\ref{eq:projChi10}) has to vanish. Consider the case $y_0=0$.
With $y_0=0$, equation (\ref{eq:projCoord}) becomes $y_1y_2-y_2y_3=0$. Plugging this into equation (\ref{eq:projE6}), one obtains
\begin{align*}
6y^2_3y^2_4(y_1+y_2+y_3+y_4)^2.
\end{align*}
For $p\neq 2,3$ assume $y_3=0$ and, therefore, $y_1=0$. The remaining two equations (\ref{eq:projE4}) and (\ref{eq:projChi12}) then give $y_2=y_4=0$. The other cases can be treated equivalently and one sees that the modular forms have empty intersection in the finite fibres, except for $p=3$.

For $p=3$, the system of equations has the six solutions
\begin{align*}
&(0:0:1:0:1),(0:0:1:1:0),(0:1:0:0:1),\\
&(0:1:0:1:0),(1:0:0:0:1),(1:0:0:1:0).
\end{align*}
An element $M\in\mathrm{Sp}_4(\mathbbm{Z}/2\mathbbm{Z}))$ acts on the $\vartheta$-functions via
$$
\vartheta_{M\cdot m}=c(M,m)\mathrm{det}(C\tau+D)^{-1}\vartheta_m\quad (m \in (\mathbbm{Z}/2\mathbbm{Z})^4)
$$
with $c(M,m)$ an 8th root of unity. We find that under this action, all above points are equivalent.
\end{proof}

 \begin{mthm}\label{thm_main}
The arithmetic self intersection number, i.e., the arithmetic degree of the line bundle $\mathcal{M}_k(\Gamma_2)$ of modular forms of weight $k$ on $\overline{\Ab}_2$, equipped with the Petersson metric, is given as
\begin{align*}
\widehat\deg(\mathcal{M}_k(\Gamma_2), \Vert\cdot\Vert_{\mathrm{Pet}})
=k^4\left(\zeta(-3)\zeta(-1)\left(2\frac{\zeta'(-3)}{\zeta(-3)}+2\frac{\zeta'(-1)}{\zeta(-1)}+\frac{17}{6}\right)+c_2\log 2 +c_3\log 3\right), 
\end{align*}
with $c_2, c_3\in \mathbbm{Q}$.
 \end{mthm} 

\begin{proof}
The main theorem follows from computing the complex contribution by adding the terms $(A)$ and $(B)$ from Proposition~\ref{prop_5integrals} and considerations about the finite contribution. As the boundary integral along $\partial B_\eps(\partial \Ab_2)$ vanishes in the limit $\eps\rightarrow 0$ by Theorem~\ref{thm_fullboundary}, the value of $(A)$ is given by
\begin{equation*}
(A)=10\cdot 6\cdot 4\cdot 12\,\zeta(-3)\zeta(-1)\left(\frac{4}{3}+2\frac{\zeta'(-3)}{\zeta(-3)}-\frac{\zeta'(-1)}{\zeta(-1)}+\frac{6}{5}\log 2\right),
\end{equation*}
according to Proposition~\ref{prop_higherdim}. The boundary integral along $\partial B_\eps(\partial \mathcal{H})$ vanishes by Proposition~\ref{prop_boundaryH}, therefore, the value of $B$ is given by
\begin{equation*}
B=-6\biggl(\frac{1}{2}+\frac{\zeta'(-1)}{\zeta(-1)}\biggr)-\frac{4}{3}\log 2-\frac{2}{3}\log 3,
\end{equation*}
according to Proposition~\ref{prop_lowerdim}. We obtain
\begin{align*}
\widehat{\vol}(\overline\Ab_2)_\infty&=\frac{1}{10\cdot 6\cdot 4\cdot 12}((A)+(B))\\
&=\zeta(-1)\zeta(-3)\left(\frac{17}{6}+2\frac{\zeta'(-3)}{\zeta(-3)}+2\frac{\zeta'(-1)}{\zeta(-1)}\right)-\frac{56}{15}\log 2 -\frac{2}{3}\log 3,
\end{align*}
as claimed.
\end{proof}

\def\bibindent{.1mm}

\end{document}